\documentclass[amstex,10pt,reqno]{amsart}
\usepackage{amsmath,amsfonts,amssymb,amsthm,enumerate,hyperref,multicol}
\newtheorem{lemma}{Lemma}
\newtheorem{thm}{Theorem}

\newtheorem{corollary}{Corollary}

\numberwithin{equation}{section}
\begin{document}

\leftline{ \scriptsize \it}
\title[]
{Approximation by Complex Sz\'{a}sz-Durrmeyer-Chlodowsky Operators in compact disks}
\maketitle

\begin{center}
{\bf Meenu Goyal$^1,$ and P. N. Agrawal$^2$}
\vskip0.2in
$^{1,2}$Department of Mathematics\\
Indian Institute of Technology Roorkee\\
Roorkee-247667, India\\

$^1$meenu.goyaliitr@yahoo.com
$^2$pna\_{}iitr@yahoo.co.in

\end{center}

\begin{abstract}
In the present article, we deal with the overconvergence of the Sz\'{a}sz-Durrmeyer-Chlodowsky operators. Here we study the approximation properties e.g. upper estimates, Voronovskaja type result for these operators attached to analytic functions in compact disks. Also, we discuss the exact order in simultaneous approximation by these operators and its derivatives and the asymptotic result with quantitative upper estimate. In such a way, we put in evidence the overconvergence phenomenon for the Sz\'{a}sz-Durrmeyer-Chlodowsky operators, namely the extensions of approximation properties with exact quantitative estimates and orders of these convergencies to sets in the complex plane that contain the interval $[0,\infty)$.\\
Keywords: Approximation in compact disks, Voronovskaja theorem, simultaneous approximation.\\
Mathematics Subject Classification(2010): 30E10, 41A25, 41A28.
\end{abstract}

\section{introduction}
For the convergence of the Bernstein polynomials in the complex plane, Bernstein \cite{GGL} proved that if $f : S\rightarrow \mathbb{C}$ is analytic in the open set $S \subset \mathbb{C},$ with $\overline{\mathbb{D}}_1\subset S$ (with $\overline{\mathbb{D}}_1 =\{z\in \mathbb{C} : |z| \leq 1\}),$ then the complex Bernstein polynomials $\mathcal{B}_{n}(f;z)=\displaystyle\sum_{k=0}^{n} {n\choose k} z^k (1-z)^{n-k} f\left(\frac{k}{n}\right),$ uniformly converges to $f$ in $\overline{\mathbb{D}}_1.$ In 1986, Lorentz \cite{GGL} was the first person who studied some approximation properties of complex Bernstein polynomials in compact disks. Nowadays, approximation by complex linear operators has become an interesting and important topic of research. A Voronovskaja type result with quantitative estimate for complex Bernstein polynomials in compact disks was studied by Gal \cite{GAL2}. After that, similar results for different operators e.g. complex Bernstein-Kantorovich, Bernstein-Stancu, complex Bernstein-Schurer and Kantorovich-Schurer polynomials, Bernstein-Durrmeyer, complex genuine Durrmeyer-Stancu, complex Baskakov-Stancu operators, complex Beta operators of first kind, Sz$\acute{a}$sz-Durrmeyer operators etc. have been discussed in literature (see \cite{GAA}, \cite{GAL10}-\cite{GAL11}, \cite{GAL1}-\cite{VRP}). In \cite{NI}, Ispir introduced the modified complex Sz\'{a}sz-Mirakjan operators and gave the order of approximation, Voronovskaja-type theorem and exact orders of approximation.
For a real function of real variable $f : [0,\infty)\rightarrow \mathbb{R},$ \.{I}zg\.{i} \cite{AI} introduced the following composition of Sz\'{a}sz-Mirakjan operators by taking the weight function of Chlodowsky-Durrmeyer operators on $C[0,\infty)$ as
\begin{eqnarray}
\left( \mathcal{F}_{n}f\right) (x) &=&\frac{n+1}{b_n}\sum_{k=0}^{\infty}p_{n,k}\left(\frac{x}{b_n}\right)\int_{0}^{b_n} \phi_{n,k}\left(\frac{t}{b_n}\right) f(t) dt, \, 0\leq x\leq b_n \label{eq1},
\end{eqnarray}
where $p_{n,k}(x)= e^{-nx}\dfrac{(nx)^k}{k!}, \, \phi_{n,k}(t)= \displaystyle{n\choose k} t^k (1-t)^{n-k}$ and $b_n$ is a sequence of positive real numbers which satisfy $\displaystyle\lim_{n\rightarrow \infty} b_n = \infty, \, \displaystyle\lim_{n\rightarrow \infty} \frac{b_n}{n}=0.$


In the present paper, we extend some overconvergence properties of the Sz\'{a}sz-Durrmeyer-Chlodowsky operators to complex domain. The complex Sz\'{a}sz-Durrmeyer-Chlodowsky operators are obtained from the real version, simply by replacing the real variable $x$ by the complex variable $z$ in the operators defined by (\ref{eq1}), which is given below:
\begin{eqnarray}
\left( \mathcal{F}_{n}f\right) (z) &=&\frac{n+1}{b_n}\sum_{k=0}^{\infty}p_{n,k}\left(\frac{z}{b_n}\right)\int_{0}^{b_n} \phi_{n,k}\left(\frac{t}{b_n}\right) f(t) dt, \label{eq2}
\end{eqnarray}
where $z\in \mathbb{C}$ is such that $0\leq Re(z)\leq b_n.$\\

Throughout the paper, we consider $\mathbb{D}_{R}:=\{z\in \mathbb{C}:|z|<R\},\, R>1.$ By $H_R,$ we mean the class of all functions satisfying: $f: [R,b_n]\cup  \mathbb{\overline D}_R\rightarrow \mathbb{C}$ is continuous in $[R,b_n]\cup  \mathbb{\overline D}_R,$ analytic in $\mathbb{D}_R$ i.e. $f(z)=\displaystyle \sum_{p=0}^\infty c_p z^p$ for all $z\in \mathbb{D}_R.$ Let $1\leq r< R$ and $\parallel f\parallel_r=\displaystyle\sup_{|z|\leq r}|f(z)|.$
In this paper, we present rate of convergence and the Voronovskaja type result for the Sz\'{a}sz-Durrmeyer-Chlodowsky operators $ \mathcal{F}_{n}(f;z)$ for analytic functions on compact disks and also study the exact order of approximation for these operators.

\section{Auxiliary results}
In order to obtain the main results, we first prove basic lemmas:
\begin{lemma}\label{l1}
Denoting $e_p(z)=z^p$ and $\Pi_{n,p}(z)= \mathcal{F}_{n}(e_p;z),$ for all $e_p=t^p,\, p\in \mathbb{N}\cup\{0\}, n\in \mathbb{N}$ and $z\in \mathbb{C},$ we have $\mathcal{F}_{n}(e_0;z)=1$ and
\begin{eqnarray*}
\Pi_{n,p+1}(z)= \frac{b_n z}{n+p+2}\, \Pi^{\prime}_{n,p}(z)+\frac{nz+(p+1)b_n}{n+p+2}\, \Pi_{n,p}(z).
\end{eqnarray*}
Also, $\Pi_{n,p}(z)$ is a polynomial of degree $p.$
\end{lemma}
\begin{proof}
Using $b_n z \,\,p_{n,k}^{\prime}\left(\dfrac{z}{b_n}\right)= (k b_n-nz)\,\,p_{n,k}\left(\dfrac{z}{b_n}\right),$ we have
\begin{eqnarray*}
b_n z\,\, \Pi^{\prime}_{n,p}(z) &=&\frac{n+1}{b_n} \sum_{k=0}^{\infty} b_n z \,\,p^{\prime}_{n,k}\left(\dfrac{z}{b_n}\right) \int_0^{b_n} \phi_{n,k}\left(\dfrac{t}{b_n}\right) t^p \,dt\\
&=& \frac{n+1}{b_n} \sum_{k=0}^{\infty} (k b_n-nz)\,\,p_{n,k}\left(\dfrac{z}{b_n}\right) \int_0^{b_n} \phi_{n,k}\left(\dfrac{t}{b_n}\right) t^p \,dt\\
&=&\frac{n+1}{b_n} \sum_{k=0}^{\infty} p_{n,k}\left(\dfrac{z}{b_n}\right) \int_0^{b_n} \{(k+1) b_n-(n+1)t+(n+1)t-b_n-nz)\}\,\,\phi_{n,k}\left(\dfrac{t}{b_n}\right) t^p \,dt.
\end{eqnarray*}
Using the identity $(b_n-t)\, \left(t\phi_{n,k}\left(\dfrac{t}{b_n}\right)\right)^{\prime} =\{(k+1)b_n-(n+1)t\}\,\left(\phi_{n,k}\left(\dfrac{t}{b_n}\right)\right),$ we obtain
\begin{eqnarray*}
b_n z\,\, \Pi^{\prime}_{n,p}(z) &=&\frac{n+1}{b_n} \sum_{k=0}^{\infty} p_{n,k}\left(\dfrac{z}{b_n}\right) \int_0^{b_n} \{(k+1)b_n-(n+1)t\}\,\left(\phi_{n,k}\left(\frac{t}{b_n}\right)\right) t^p \,dt\\&&+(n+1) \,\Pi_{n,p+1}(z)-(nz+b_n)\,\Pi_{n,p}(z)\\
&=&\frac{n+1}{b_n} \sum_{k=0}^{\infty} p_{n,k}\left(\dfrac{z}{b_n}\right) \int_0^{b_n} (b_n-t)\, \left(t\phi_{n,k}\left(\frac{t}{b_n}\right)\right)^{\prime}\, t^p \,dt\\&&+(n+1) \,\Pi_{n,p+1}(z)-(nz+b_n)\,\Pi_{n,p}(z)
\end{eqnarray*}
Thus integrating by parts on the right side, we get
\begin{eqnarray*}
b_n z \,\Pi^{\prime}_{n,p}(z)&=& -p b_n\, \Pi_{n,p}(z)+(p+1)\,\Pi_{n,p+1}(z)+(n+1) \,\Pi_{n,p+1}(z)-(nz+b_n)\,\Pi_{n,p}(z)\\
&=& (n+p+2)\,\Pi_{n,p+1}(z)-(nz+(p+1)b_n)\,\Pi_{n,p}(z),
\end{eqnarray*}
which completes the proof of the recurrence relation. Further, by mathematical induction on $p,$ we easily get that $\Pi_{n,p}(z)$ is a polynomial of degree $p.$
\end{proof}

\begin{lemma}\label{l2}
Let $f\in H_R$ and be bounded and integrable on $[0,b_n].$ Suppose that $f(z)=\displaystyle \sum_{p=0}^\infty c_p z^p$ for all $z\in \mathbb{D}_R$ and $1\leq r< R.$ Then for all $|z|\leq r$ and $n\in \mathbb{N},$ we have
\begin{eqnarray*}
\mathcal{F}_{n}(f;z)=\sum_{p=0}^{\infty}c_p \mathcal{F}_{n}(e_p;z).
\end{eqnarray*}
\end{lemma}
\begin{proof}
For any $m\in \mathbb{N}$ and $r<R,$ we define $f_m(z)=\displaystyle \sum_{p=0}^m c_pz^p$ if $|z|\leq r$ and $f_m(x) =f(x)$ if $x\in (r,b_n].$
Since $|f_m(z)|\leq \displaystyle\sum_{p=0}^\infty |c_p|r^p=C_r$ for $|z|\leq r $ and $m\in \mathbb{N}$ and $f_m$ is bounded and integrable on $[0,b_n],$
\begin{eqnarray*}
\left|\mathcal{F}_{n}(f_m;z)\right| &\leq&\frac{n+1}{b_n}\sum_{k=0}^{\infty} \left|p_{n,k}\left(\frac{z}{b_n}\right)\right| \int_{0}^{b_n} \phi_{n,k}\left(\frac{t}{b_n}\right) \left|f_m(t)\right| dt<\infty.
\end{eqnarray*}
Thus, $\mathcal{F}_{n}(f_m;z)$ is well defined and an analytic function of $z.$\\
Similarly, for the function $f,$ it follows that $\mathcal{F}_{n}(f;z)$ is also well defined and it is an analytic function of $z.$\\
Further, we assume that $f_{m,p}(z)=c_p e_p(z)$ if $|z|\leq r$ and $f_{m,p}(x)=\frac{f(x)}{m+1}$ if $x\in (r,b_n].$ Let $|z|\leq r $ and $1\leq r< R.$\\
Defining $f_m(z)=\displaystyle\sum_{p=0}^{m}f_{m,p}(z),$ by the linearity of $\mathcal{F}_{n},$ it follows that
\begin{eqnarray*}
\mathcal{F}_{n}(f_m;z)=\sum_{p=0}^{m}c_p \mathcal{F}_{n}(e_p;z)\,\,\, \mbox{for all}\,\,\, |z|\leq r\,\, \mbox{and}\,\, m,\,n \in \mathbb{N}.
\end{eqnarray*}
It is sufficient to prove that for any fixed $n\in \mathbb{N}$
\begin{eqnarray*}
\lim_{m\rightarrow \infty}\mathcal{F}_{n}(f_m;z)= \mathcal{F}_{n}(f;z),
\end{eqnarray*}
uniformly in compact disk $|z|\leq r.$\\
But this is immediate from $\displaystyle\lim_{m\rightarrow\infty}\Vert f_m-f\Vert_r=0$, from $\Vert f_m-f\Vert_{B[0,\infty)}\leq\Vert f_m-f\Vert_r$ and from the inequality \\
\noindent
$|\mathcal{F}_{n}(f_m;z)-\mathcal{F}_{n}(f;z)|$
\begin{eqnarray*}
&\leq& \frac{n+1}{b_n}\sum_{k=0}^{\infty} \left|p_{n,k}\left(\frac{z}{b_n}\right)\right|\int_{0}^{b_n} \phi_{n,k}\left(\frac{t}{b_n}\right) |f_m(t)-f(t)| dt\\
&\leq& \parallel f_m-f\parallel_r \sum_{k=0}^{\infty} \left|p_{n,k}\left(\frac{z}{b_n}\right)\right|\\
&\leq& \parallel  f_m-f\parallel_r \sum_{k=0}^{\infty}\left|\frac{e^{\frac{-nz}{b_n}}\left(\frac{nz}{b_n}\right)^k}{k!}\right|\\
&\leq& \parallel  f_m-f\parallel_r \left|e^{\frac{-nz}{b_n}}\right| \sum_{k=0}^{\infty} \frac{(n|z|)^k}{k! (b_n)^k}\\
&\leq& \parallel  f_m-f\parallel_r \left|e^{\frac{-nz}{b_n}}\right| e^{\frac{n|z|}{b_n}}\leq M_{r,n}  \parallel  f_m-f\parallel_r
\end{eqnarray*}
valid for all $|z|\leq r,$ where $\parallel . \parallel_{B[0,\infty)}$ denotes the uniform norm on $[0,\infty).$ Thus, as $m\rightarrow \infty,$ we get the required result.
\end{proof}

\section{Main results}
\subsection{Upper estimates}
In the following theorem, we obtain an upper estimate of the error in the approximation of an analytic function by the operators (\ref{eq2}) in a compact disk.
\begin{thm}\label{t1}
Let $f:[R,b_n]\cup \overline{\mathbb{D}}_R\rightarrow \mathbb{C}$ be continuous in $[R,b_n]\cup \overline{\mathbb{D}}_R$ and analytic in $\mathbb{D}_R.$ Further, let $f$ be bounded and integrable in $[0,b_n].$ Suppose that there exists $M>0$ and $A\in \left(\frac{1}{R},1\right)$ with the property $|c_p|\leq \dfrac{M A^p}{(2p)!},\,\, \forall \,\,p\in\mathbb{N}\cup\{0\}.$ Let $1\leq r<\frac{1}{A}$ be arbitrary but fixed then for all $|z|\leq r$ and
$n \geq n_0,\, n_0\in \mathbb{N},$ we have
\begin{eqnarray*}
|\mathcal{F}_{n}(f;z)-f(z)|\leq C_{r,A}(f) \frac{b_n+1}{n+2}, \,\,\,\, \mbox{where}\,\,\,\, C_{r,A}(f)=M  \displaystyle\sum_{p=1}^{\infty} (A r)^{p}< \infty.
\end{eqnarray*}
\end{thm}
\begin{proof}
By using the recurrence relation of Lemma \ref{l1}, we have
\begin{eqnarray*}
\Pi_{n,p+1}(z)= \frac{b_n z}{n+p+2}\, \Pi^{\prime}_{n,p}(z)+\frac{nz+(p+1)b_n}{n+p+2}\, \Pi_{n,p}(z), \, \forall z\in \mathbb{C}, p\in \mathbb{N}\cup \{0\}, n\in \mathbb{N}.
\end{eqnarray*}
From this we immediately get the recurrence formula
\begin{eqnarray*}
\Pi_{n,p}(z)-z^p &= &\frac{b_n z}{n+p+1}\, \left(\Pi_{n,p-1}(z)-z^{p-1}\right)^{\prime}+\frac{nz+p b_n}{n+p+1}\, \left(\Pi_{n,p-1}(z)-z^{p-1}\right)+\frac{(2p-1)b_n-(p+1)z}{n+p+1} z^{p-1},  \, \\&&\forall z\in \mathbb{C}, p, n \in \mathbb{N}.
\end{eqnarray*}
Now, for $1\leq r<R,$ by linear transformation the Bernstein's inequality in the closed unit disk becomes $P_p^{\prime}(z)\leq \frac{p}{r}\|P_p\|_r,$ for all $|z|\leq r,$ where $P_p(z)$ is a polynomial of degree $\leq p.$ Thus, from the above recurrence relation, we get
\begin{eqnarray*}
\|\Pi_{n,p}-e_p\|_r &\leq& \frac{b_n r}{n+p+1}\, \|\Pi_{n,p-1}-e_{p-1}\|_r \frac{p-1}{r}+\frac{nr+p b_n}{n+p+1}\, \|\Pi_{n,p-1}-e_{p-1}\|_r+\frac{(2p-1)b_n+(p+1)r}{n+p+1} r^{p-1}\\
&\leq& \left(r+\frac{(2p-1)b_n}{n+2}\right)\|\Pi_{n,p-1}-e_{p-1}\|_r+\frac{(2p-1)b_n+(p+1)r}{n+2}r^{p-1}\\
&\leq& \left(r+\frac{(2p-1)b_n}{n+2}\right)\|\Pi_{n,p-1}-e_{p-1}\|_r+\frac{2p (b_n+1)r^{p}}{n+2}.
\end{eqnarray*}
In what follows we prove the result by mathematical induction with respect to $p,$ that this recurrence relation implies
\begin{eqnarray*}
\|\Pi_{n,p}-e_p\|_r &\leq& \frac{(2p)! r^p (b_n+1)}{n+2},\,\, \mbox{for all}\,\, p \in \mathbb{N},\, n\geq n_0, n_0\in \mathbb{N}.
\end{eqnarray*}
Indeed for $p=1$ and $n\geq n_0, n_0\in \mathbb{N},$ the left hand side is $\dfrac{b_n+2r}{n+2}$ and the right hand side is $\dfrac{2r(b_n+2)}{n+2}.$ Suppose that it is valid for $p,$ the above recurrence relation implies that
\begin{eqnarray*}
\|\Pi_{n,p+1}-e_{p+1}\|_r &\leq& \left(r+\frac{(2p+1)b_n}{n+2}\right)\frac{(2p)! r^p (b_n+1)}{n+2}+ (2p+2)r^{p+1} \frac{b_n+1}{n+2},
\end{eqnarray*}
it remains to prove that
\begin{eqnarray*}
\left(r+\frac{(2p+1)b_n}{n+2}\right)\frac{(2p)! r^p (b_n+1)}{n+2}+ (2p+2)r^{p+1}\frac{b_n+1}{n+2} \leq \frac{(2p+2)! r^{p+1} (b_n+1)}{n+2}
\end{eqnarray*}
or
\begin{eqnarray*}
\left(r+\frac{(2p+1)b_n}{n+2}\right)(2p)!+(2p+2)r&\leq &(2p+2)! r.
\end{eqnarray*}
It is easy to see by mathematical induction that this last inequality holds true for all $p\geq 1$ and $n\geq n_0, n_0\in \mathbb{N}.$ From the hypothesis on $f,$ by Lemma \ref{l2} we can write
\begin{eqnarray*}
\mathcal{F}_{n}(f;z)=\sum_{p=0}^{\infty}c_p \mathcal{F}_{n}(e_p;z)=\sum_{p=0}^{\infty}c_p \Pi_{n,p}(z),\,\, \mbox{for all}\,\, z\in \mathbb{D}_R,\, n\in \mathbb{N},
\end{eqnarray*}
which from the hypothesis on $c_p$ immediately implies for all $|z|\leq r$ with $Re(z)\leq b_n, n\in \mathbb{N}$ with $n\geq n_0, n_0\in \mathbb{N}$
\begin{eqnarray*}
|\mathcal{F}_{n}(f;z)-f(z)|\leq \sum_{p=1}^{\infty}|c_p| |\Pi_{n,p}(z)-e_p(z)|\leq \sum_{p=1}^{\infty} \frac{M (Ar)^p (b_n+1)}{n+2}=C_{r,A}(f)\frac{(b_n+1)}{n+2},
\end{eqnarray*}
where $C_{r,A}(f)=M  \displaystyle\sum_{p=1}^{\infty} (Ar)^{p}< \infty $ for all $1\leq r<\frac{1}{A},$ by ratio test. Thus, the proof is completed.
\end{proof}

\subsection{Voronovskaja-type result}
In the following theorem we obtain a quantitative Voronovskaja-type result:
\begin{thm}\label{t2}
Let $f\in H_R$ and be bounded and integrable on $[0,b_n]$ and suppose that there exists $M>0$ and $A\in \left(\frac{1}{R},1\right)$ with the property $|c_p|\leq \dfrac{M A^p}{(2p)!}.$ Let $1\leq r<\frac{1}{A}$ be arbitrary but fixed then for all $|z|\leq r$ and $p\in \mathbb{N}\cup \{0\}, \, n\geq n_0,\,\, n_0\in \mathbb{N},$ we have
\begin{eqnarray*}
\left|\mathcal{F}_{n}(f;z)-f(z)-\frac{b_n}{n+2}\left(\left(1-\frac{2z}{b_n}\right)f^{\prime}(z)+z\left(1-\frac{z}{2 b_n}\right) f^{\prime\prime}(z)\right)\right|\leq L_{r,A}(f) \frac{(b_n+1)^2}{(n+2)^2},
\end{eqnarray*}
where $L_{r,A}(f)= \dfrac{2 M }{(1-Ar)\log\frac{1}{Ar}}+4M\displaystyle\sum_{p=1}^{\left[\alpha\right]}p(Ar)^{p}<\infty.$
\end{thm}
\begin{proof}
By using Lemma \ref{l2}, we may write $\mathcal{F}_{n}(f;z)=\displaystyle\sum_{p=0}^{\infty} c_p \mathcal{F}_{n}(e_p;z)$ and
\begin{eqnarray*}
\frac{b_n}{n+2}\left(\left(1-\frac{2z}{b_n}\right)f^{\prime}(z)+z\left(1-\frac{z}{2 b_n}\right) f^{\prime\prime}(z)\right)=\frac{b_n}{n+2} \sum_{p=1}^{\infty} c_p \left(p^2 z^{p-1}-\frac{p^2+3p}{2b_n}z^p\right).
\end{eqnarray*}
Defining $\Pi_{n,p}(z)=\mathcal{F}_{n}(e_p)(z),$ we get
\noindent\\
$\left|\mathcal{F}_{n}(f;z)-f(z)-\dfrac{b_n}{n+2}\left(\left(1-\dfrac{2z}{b_n}\right)f^{\prime}(z)+z\left(1-\dfrac{z}{2 b_n}\right) f^{\prime\prime}(z)\right)\right|$
\begin{eqnarray*}
&\leq& \sum_{p=1}^\infty |c_p| \left|\Pi_{n,p}(z)-e_p(z)-\frac{b_n}{n+2}\left(p^2 e_{p-1}(z)-\frac{p^2+3p}{2b_n}e_p(z)\right)\right|,
\end{eqnarray*}
for all $z\in \mathbb{D}_R, \, n\in \mathbb{N}.$
Now, by applying Lemma \ref{l1}, we get the following recurrence relation
\begin{eqnarray*}
\Pi_{n,p}(z)= \frac{b_n z}{n+p+1}\, \Pi^{\prime}_{n,p-1}(z)+\frac{nz+p b_n}{n+p+1}\, \Pi_{n,p-1}(z).
\end{eqnarray*}
Let us denote
\begin{eqnarray*}
\mathcal{E}_{n,p}(z)=\Pi_{n,p}(z)-e_p(z)-\frac{b_n}{n+2}\left(p^2 e_{p-1}(z)-\frac{p^2+3p}{2b_n}e_p(z)\right).
\end{eqnarray*}
Then,
\begin{eqnarray}\label{eq3}
\mathcal{E}_{n,p}(z)&=&\frac{b_n z}{n+p+1} \mathcal{E}_{n,p-1}^{\prime}(z)+\frac{nz+p b_n}{n+p+1}\mathcal{E}_{n,p-1}(z)+ \mathcal{X}_{n,p}(z),
\end{eqnarray}
where\\
\noindent $\mathcal{X}_{n,p}(z)$
\begin{eqnarray*}
&=& \frac{1}{2(n+p+1)(n+2)}\bigg(z^{p-2}b_n^2 \{2(p-1)^2(p-2)+2p(p-1)^2\}\\&&
+z^{p-1}b_n\{(n+2)(4p-2)+2n(p-1)^2-2(n+p+1)p^2-(p-1)(2p^2+p-2)\}\\&&
+z^p\{2n(n+2)-2(n+p+1)(n+2)-n(p-1)^2-3n(p-1)+(p^2+3p)(n+p+1)\}\bigg)\\
&=& \frac{1}{2(n+p+1)(n+2)}\bigg(z^{p-2}b_n^2 \{4(p-1)^3\}-z^{p-1}b_n\{4p^3+p^2-10p+6\}+z^p\{(p+1)(p^2+3p-4)\}\bigg).
\end{eqnarray*}
Hence
\begin{eqnarray}\label{eq4}
|\mathcal{X}_{n,p}(z) |&\leq& \frac{2(b_n+1)^2 (p+1)^3}{(n+2)^2} r^{p},\,\, \forall \,\,n\in \mathbb{N}.
\end{eqnarray}

It is immediate that $\mathcal{E}_{n,p}(z)$ is a polynomial in $z$ of degree $\leq p$ and that $\mathcal{E}_{n,0}(z)=0.$ Combining (\ref{eq3}) and (\ref{eq4}), we have
\begin{eqnarray*}
|\mathcal{E}_{n,p}(z)|&\leq& \frac{b_n r}{n+2}|\mathcal{E}_{n,p-1}^\prime(z)|+\left(r+\frac{p b_n }{n+2}\right)|\mathcal{E}_{n,p-1}(z)|+\frac{2(b_n+1)^2 (p+1)^3}{(n+2)^2} r^{p}.
\end{eqnarray*}
Now, we shall find the estimate of $\mathcal{E}_{n,p-1}^\prime(z)$ for $p\geq 1.$ Taking into account the fact that $\mathcal{E}_{n,p-1}(z)$ is a polynomial of degree $\leq p-1,$ we have
\begin{eqnarray*}
|\mathcal{E}_{n,p-1}^{\prime}(z)|&\leq& \frac{p-1}{r}\|\mathcal{E}_{n,p-1}\|_r\leq \frac{p-1}{r} \bigg(\left\|\Pi_{n,p-1}-e_{p-1}\right\|_r\\&&+\left\|\frac{b_n}{n+2}\left((p-1)^2 e_{p-2}(z)-\frac{(p-1)^2+3(p-1)}{2b_n}e_{p-1}(z)\right)\right\|_r\bigg)\\
&\leq& \frac{p-1}{r} \left(\frac{(2p-2)! r^{p-1}(b_n+1)}{n+2}+\frac{2(p-1)^2b_n+(p-1)(p+2)}{2(n+2)}r^{p-1}\right)\\
&\leq& \frac{2(2p-2)! (p-1)r^{p-2}(b_n+1)}{n+2},\,\, \forall \,\, n\in \mathbb{N}.
\end{eqnarray*}
Thus
\begin{eqnarray*}
\frac{r b_n}{n+2}|\mathcal{E}_{n,p-1}^{\prime}(z)|&\leq& \frac{2(2p-2)! (p-1)r^{p-1}(b_n+1)^2}{(n+2)^2},\,\, \forall \,\, n\in \mathbb{N}
\end{eqnarray*}
and \\ \noindent $|\mathcal{E}_{n,p}(z)|$
\begin{eqnarray*}
&\leq& \left(r+\frac{p b_n }{n+2}\right)|\mathcal{E}_{n,p-1}(z)|+\frac{2(2p-2)! (p-1)r^{p-1}(b_n+1)^2}{(n+2)^2}+\frac{2(b_n+1)^2 (p+1)^3}{(n+2)^2} r^{p}\\
&\leq&  \left(r+\frac{p b_n }{n+2}\right)|\mathcal{E}_{n,p-1}(z)|+ \frac{4(2p)! r^{p}(b_n+1)^2}{(n+2)^2},\,\, \mbox{for all}\,\, |z|\leq r\,\, \mbox{and}\,\, n\in \mathbb{N}.
\end{eqnarray*}
For $1\leq p \leq \dfrac{n+2}{b_n}=\alpha \,(say)$ and $|z|\leq r,$ taking into account that $ \left(r+p \alpha\right)\leq (r+1),$ we have
\begin{eqnarray*}
|\mathcal{E}_{n,p}(z)|&\leq& (r+1)|\mathcal{E}_{n,p-1}(z)|+  \frac{4(2p)! r^{p}(b_n+1)^2}{(n+2)^2}.
\end{eqnarray*}
But $\mathcal{E}_{n,0}(z)=0,$ for any $z\in \mathbb{C},$ therefore by writing the inequality for $1\leq p\leq \alpha,$ we easily obtain step by step the following
\begin{eqnarray*}
|\mathcal{E}_{n,p}(z)|&\leq& \frac{4 r^{p}(b_n+1)^2}{(n+2)^2}\sum_{j=1}^{p} (2j)!\leq \frac{4 p (2p)!r^{p} (b_n+1)^2}{(n+2)^2}.
\end{eqnarray*}
Denoting by $\left[\alpha\right]$ the integral part of $\alpha,$ it follows that\\
\noindent $\left|\mathcal{F}_{n}(f;z)-f(z)-\dfrac{b_n}{n+2}\left(\left(1-\dfrac{2z}{b_n}\right)f^{\prime}(z)+z\left(1-\dfrac{z}{2 b_n}\right) f^{\prime\prime}(z)\right)\right|$
\begin{eqnarray*}
&\leq &\sum_{p=1}^{\left[\alpha\right]} |c_p| |\mathcal{E}_{n,p}(z)|+\sum_{p=\left[\alpha\right]+1}^\infty |c_p| |\mathcal{E}_{n,p}(z)|\\
&\leq& \sum_{p=1}^{\left[\alpha\right]} |c_p| \frac{4 p (2p)!r^{p} (b_n+1)^2}{(n+2)^2}+\sum_{p=\left[\alpha\right]+1}^\infty |c_p| |\mathcal{E}_{n,p}(z)|\\
&\leq& \frac{4M(b_n+1)^2}{(n+2)^2}\sum_{p=1}^{\left[\alpha\right]} p (Ar)^{p}+\sum_{p=\left[\alpha\right]+1}^\infty |c_p| |\mathcal{E}_{n,p}(z)|.
\end{eqnarray*}
But
\begin{eqnarray*}
\sum_{p=\left[\alpha\right]+1}^\infty |c_p| |\mathcal{E}_{n,p}(z)|&\leq& \sum_{p=\left[\alpha\right]+1}^\infty |c_p| \left(\left|\Pi_{n,p}(z)-e_{p}(z)\right|+\frac{b_n}{n+2}\left|p^2 e_{p-1}(z)-\frac{p^2+3p}{2b_n}e_{p}(z)\right|\right)\\
&\leq& \sum_{p=\left[\alpha\right]+1}^\infty |c_p| \left(\left|\frac{(2p)! r^p (b_n+1)}{n+2}\right|+\frac{r^{p}}{n+2}\left|p^2b_n +\frac{p(p+3)}{2} \right|\right)\\
&\leq&  2\sum_{p=\left[\alpha\right]+1}^\infty |c_p| \frac{(2p)! r^{p}(b_n+1)}{n+2}\\
&\leq& \frac{2M(b_n+1)}{(n+2)}  \sum_{p=\left[\alpha\right]+1}^\infty (Ar)^{p} \leq \frac{2 M (b_n+1)}{(n+2)} \frac{(Ar)^{\alpha}}{(1-Ar)}, \forall n\geq n_0, \, n_0\in \mathbb{N}.
\end{eqnarray*}
Also, by $e^t=1+t+\frac{t^2}{2}+\cdot \cdot \cdot,$ we get $e^t\geq t \, \,\,\forall\,\, t\geq 0,$ which combined with \\
$\dfrac{1}{(Ar)^{\alpha}}= e^{\alpha\log\left(\frac{1}{Ar}\right)} \Rightarrow \dfrac{1}{(Ar)^{\alpha}}\geq \alpha \log\frac{1}{Ar},$ for all $\alpha >0.$ So, $(Ar)^\alpha\leq \frac{1}{\alpha \log\frac{1}{Ar}}.$\\
Therefore, we get
\begin{eqnarray*}
\sum_{p=\left[\alpha\right]+1}^\infty |c_p| |\mathcal{E}_{n,p}(z)| \leq \frac{2 M (b_n+1)^2}{(n+2)^2(1-Ar)\log\frac{1}{Ar}},\,\, \mbox{for all}\,\, |z|\leq r\,\, \mbox{and}\,\, n\geq n_0, \ n_0\in \mathbb{N}.
\end{eqnarray*}
Finally, we obtain\\
\noindent
$\left|\mathcal{F}_{n}(f;z)-f(z)-\dfrac{b_n}{n+2}\left(\left(1-\dfrac{2z}{b_n}\right)f^{\prime}(z)+z\left(1-\dfrac{z}{2 b_n}\right) f^{\prime\prime}(z)\right)\right|$
\begin{eqnarray*}
&\leq& \frac{4M(b_n+1)^2}{(n+2)^2}\sum_{p=1}^{\left[\alpha\right]} p (Ar)^{p}+\frac{2 M (b_n+1)^2}{(n+2)^2(1-Ar)\log\frac{1}{Ar}},
\end{eqnarray*}
where for $rA<1,$ by ratio test the above series is convergent. This completes the proof of the theorem.
\end{proof}

\subsection{Exact order of approximation}
To obtain the exact degree of approximation by $\mathcal{F}_{n}(f;z),$ in the following theorem we get a lower estimate of the error in the approximation of $f$ by $\mathcal{F}_{n}(f;z)$:
\begin{thm}\label{t3}
In the hypothesis of Theorem \ref{t2}, if $f$ is not a polynomial of degree $\leq 0,$ then for any $1\leq r< R,$ we have
\begin{eqnarray*}
\parallel \mathcal{F}_{n}(f)-f\parallel_r \geq \frac{b_n+1}{n+2}P_{r}(f), \,\,\,\, n\geq n_0, \,\, n_0\in \mathbb{N},
\end{eqnarray*}
where the constants in the equivalence $P_{r}(f)>0,$ depends on $f,r.$
\end{thm}
\begin{proof}
For all $|z|\leq r,$ and $n\in \mathbb{N},$ we can write the following equality
\begin{eqnarray*}
\mathcal{F}_{n}(f;z)-f(z)&=&\frac{b_n+1}{n+2}\bigg\{\bigg(z\left(1-\frac{z}{2b_n}\right) f^{\prime\prime}(z)+ \left(1-\frac{2z}{b_n}\right)f^{\prime}(z)\bigg)\\&&+\frac{b_n+1}{n+2}\bigg(\frac{(n+2)^2}{(b_n+1)^2}\bigg(\mathcal{F}_{n}(f;z)-f(z)
-\frac{b_n+1}{n+2}\left(\left(1-\frac{2z}{b_n}\right)f^{\prime}(z)+z\left(1-\frac{z}{2b_n}\right)f^{\prime\prime}(z)\right)\bigg)\bigg)\bigg\}.
\end{eqnarray*}
By applying the property
\begin{eqnarray*}
\parallel F+G\parallel_r\geq |\parallel F\parallel_r-\parallel G\parallel_r|\geq \parallel F\parallel_r-\parallel G\parallel_r,
\end{eqnarray*}
it follows that
\begin{eqnarray*}
\left\|\mathcal{F}_{n}(f;.)-f\right\|_r &\geq& \frac{b_n+1}{n+2}\bigg\{\bigg\|\left(e_1 \left(1-\dfrac{e_1}{2b_n}\right) f^{\prime\prime}+\left(1-\dfrac{2e_1}{b_n}\right)f^{\prime}\right)\bigg\|_r\\&&-\frac{b_n+1}{n+2}\left(\frac{(n+2)^2}{(b_n+1)^2}\bigg\| (\mathcal{F}_{n}(f;.)-f-\frac{b_n+1}{n+2}\left(e_1 \left(1-\dfrac{e_1}{2b_n}\right) f^{\prime\prime}+\left(1-\dfrac{2e_1}{b_n}\right)f^{\prime}\right)\bigg\|_r \right)\bigg\}.
\end{eqnarray*}
Taking into account that by hypothesis $f$ is not a polynomial of degree $\leq 0$ in $\mathbb{D}_R,$ we get
$$\left\|\left(e_1 \left(1-\dfrac{e_1}{2b_n}\right) f^{\prime\prime}+\left(1-\dfrac{2e_1}{b_n}\right)f^{\prime}\right)\right\|_r>0.$$
Indeed, supposing the contrary, it follows that
\begin{eqnarray}\label{e1}
z\left(1-\frac{z}{2b_n}\right) f^{\prime\prime}(z)+ \left(1-\frac{2z}{b_n}\right)f^{\prime}(z)=0 \,\,\, \mbox{for all}\,\, |z|\leq r.
\end{eqnarray}
Let us take $f(z)=\displaystyle\sum_{p=0}^\infty c_p z^p ,$ where $c_p, 0\leq p<\infty,$ are certain constants.\\
Then $$f^{\prime}(z)=\sum_{p=1}^\infty p c_p z^{p-1},\,\, f^{\prime\prime}(z)=\sum_{p=2}^\infty p (p-1) c_p z^{p-2}.$$
By substituting these values in (\ref{e1}), we obtain
\begin{eqnarray*}
\left(1-\frac{z}{2b_n}\right)\sum_{p=2}^\infty p (p-1) c_p z^{p-1}+\left(1-\frac{2z}{b_n}\right)\sum_{p=1}^\infty p c_p z^{p-1}=0,
\end{eqnarray*}
or
\begin{eqnarray*}
c_1+\left(4c_2-\frac{2}{b_n}c_1\right)z+\sum_{p=2}^{\infty}\left((p+1)^2c_{p+1}-\frac{p(p+3)}{2b_n}c_p\right)z^p=0,\,\, |z|\leq r.
\end{eqnarray*}
From the above series, we easily get $c_p=0,\,\, \forall p\in \mathbb{N}$ and $f(z)=c_0,$ a contradiction to hypothesis.\\
Now, from Theorem \ref{t2}, we have
\begin{eqnarray*}
\frac{(n+2)^2}{(b_n+1)^2}\bigg\|(\mathcal{F}_{n}(f;.)-f-\frac{b_n+1}{n+2}\left(e_1 \left(1-\dfrac{e_1}{2b_n}\right) f^{\prime\prime}+\left(1-\dfrac{2e_1}{b_n}\right)f^{\prime}\right)\bigg\|_r \leq L_{r,A}(f),\,\, \forall\,\, n\geq n_0,\, n_0\in \mathbb{N}.
\end{eqnarray*}
Therefore there exists an index $n^*> n_0$ depending only on $f,r$ such that $n\geq n^*$ we have\\
\noindent $\left\|e_1 \left(1-\dfrac{e_1}{2b_n}\right) f^{\prime\prime}+\left(1-\dfrac{2e_1}{b_n}\right)f^{\prime}\right\|_r$
\begin{eqnarray*}
&&-\frac{b_n+1}{n+2}\left(\frac{(n+2)^2}{(b_n+1)^2}
\left\|\mathcal{F}_{n}(f;.)-f-\frac{b_n+1}{n+2}\left(e_1 \left(1-\dfrac{e_1}{2b_n}\right) f^{\prime\prime}+\left(1-\dfrac{2e_1}{b_n}\right)f^{\prime}\right)\right\|_r\right)\\
&&\geq \frac{1}{2}\left\Vert \left(e_1 \left(1-\dfrac{e_1}{2b_n}\right) f^{\prime\prime}+\left(1-\frac{2z}{b_n}\right)f^{\prime}\right)\right\Vert_r,
\end{eqnarray*}
which immediately implies that
\begin{eqnarray*}
\left\| \mathcal{F}_{n}(f;.)-f\right\|_r \geq \frac{b_n+1}{2(n+2)}\left\Vert e_1 \left(1-\dfrac{e_1}{2b_n}\right) f^{\prime\prime}+\left(1-\dfrac{2e_1}{b_n}\right)f^{\prime}\right\Vert_r.
\end{eqnarray*}
For $n_0\leq n< n^*,$ we have
\begin{eqnarray*}
\left\| \mathcal{F}_{n}(f;.)-f\right\|_r \geq \frac{b_n+1}{(n+2)}L_{r,n}(f)
\end{eqnarray*}
with $L_{r,n}(f)= \dfrac{n+2}{b_n+1} \left\| \mathcal{F}_{n}(f;.)-f\right\|_r>0.$
Indeed, if we would have $\left\| \mathcal{F}_{n}(f;.)-f\right\|_r=0,$ it would follow that $\mathcal{F}_{n}(f;z)=f(z) \,\,\, \mbox{for all}\,\,\, |z|\leq r,$ which is valid only for $f,$ a constant function, contradicting the hypothesis on $f.$ Therefore, finally we get
\begin{eqnarray*}
\left\Vert \mathcal{F}_{n}(f;.)-f\right\Vert_r \geq \frac{b_n+1}{(n+2)}P_{r}(f),\,\,\, \mbox{for all}\,\,\, n\geq n_0,\,\, n_0 \in\mathbb{N},
\end{eqnarray*}
where
\begin{eqnarray*}
P_{r}(f)=\min \bigg\{L_{r,n_0}(f), L_{r,n_0+1}(f),\cdot\cdot\cdot L_{r,n^*-1}(f), \frac{1}{2}\left\|e_1 \left(1-\dfrac{e_1}{2b_n}\right) f^{\prime\prime}+\left(1-\dfrac{2e_1}{b_n}\right)f^{\prime}\right\|_r\bigg\}
\end{eqnarray*}
which completes the proof.
\end{proof}

Now, combining Theorem \ref{t1} and Theorem \ref{t3}, we immediately get the following:

\begin{corollary}
In the hypothesis of Theorem \ref{t2}, if $f$ is not a polynomial of degree $\leq 0,$  then for any $1\leq r<R,$ we have
\begin{eqnarray*}
\parallel \mathcal{F}_{n}(f;.)-f\parallel_r \sim \frac{b_n+1}{n+2}, \,\, n\geq n_0, \, n_0 \in \mathbb{N}
\end{eqnarray*}
holds, where the constant in the equivalence $\sim$ depends on $f,r.$
\end{corollary}

\subsection{Simultaneous approximation}
Concerning the derivatives of complex Sz\'{a}sz-Durrmeyer-Chlodowsky operators, we can prove the following results:
\begin{thm}\label{t4}
In the hypothesis of Theorem \ref{t2}, let $1\leq r< r_1<R$ and $p\in \mathbb{N},$ then for all $|z|\leq r$ and $n\geq n_0,\, n_0\in \mathbb{N},$ we have
\begin{eqnarray*}
|\mathcal{F}_{n}^{(p)}(f;z)-f^{(p)}(z)| \leq \frac{(b_n+1)\, C_{r_1,A}(f)\,\, p! \,\,r_1}{(n+2)(r_1-r)^{p+1}},
\end{eqnarray*}
where $C_{r_1,A}(f)$ is defined as in Theorem \ref{t1}.
\end{thm}
\begin{proof}
Denoting by $\Gamma,$ the circle of radius $r_1 >1$ and center $0,$ since for any $|z|\leq r$ and $\nu \in \Gamma$ we have $|\nu-z|\geq r_1-r,$ by Cauchy's formula, it follows that for all $n\in \mathbb{N},$ we get
\begin{eqnarray*}
|\mathcal{F}_{n}^{(p)}(f;z)-f^{(p)}(z)|&=&\frac{p!}{2 \pi}\left|\int_{\Gamma}\frac{\mathcal{F}_{n}(f;\nu)-f(\nu)}{(\nu-z)^{p+1}}d\nu\right|.
\end{eqnarray*}
For all $\nu\in \Gamma$ and $n\in \mathbb{N}$ with $n\geq n_0, \, n_0\in \mathbb{N},$ we get
\begin{eqnarray*}
|\mathcal{F}_{n}^{(p)}(f;z)-f^{(p)}(z)| &\leq& \frac{p! \,r_1\, (b_n+1)\, C_{r_1,A}(f)}{(n+2)\,(r_1-r)^{p+1}},
\end{eqnarray*}
which proves the theorem.
\end{proof}


\begin{thm}\label{t5}
In the hypothesis of Theorem \ref{t2}, let $1\leq r< r_1<R$ and $f$ be not a polynomial of degree $\leq p-1,\, (p\geq 1)$ then we have
\begin{eqnarray*}
\parallel \mathcal{F}_{n}^{(p)}(f;.)-f^{(p)}\parallel_r  \sim \frac{b_n+1}{n+2}, \,\, \mbox{for all}\,\,n\geq n_0,\, n_0\in \mathbb{N},
\end{eqnarray*}
where the constant in the equivalence $\sim$ depends only on $f,r,r_1,p.$
\end{thm}
\begin{proof}
Let $\Gamma$ be a circle of radius $r_1>r\geq 1$ and center $0,$ we have\\
\noindent $\left\| \mathcal{F}_{n}^{(p)}(f;.)-f^{(p)}\right\|_r$
\begin{eqnarray*}
&=&\bigg\Vert\frac{b_n+1}{(n+2)}
\bigg\{\frac{p!}{2\pi i}\int_{\Gamma}\frac{\left[\nu\left(1-\dfrac{\nu}{2b_n}\right) f^{\prime \prime}(\nu)+ \left(1-\dfrac{2\nu}{b_n}\right)f^{\prime}(\nu)\right]}{(\nu-z)^{p+1}}d\nu
\\&&+\frac{b_n+1}{(n+2)}\frac{p!}{2\pi i}\int_{\Gamma}\frac{(n+2)^2}{(b_n+1)^2}\frac{\left[ \mathcal{F}_{n}(f;\nu)-f(\nu)-\dfrac{b_n+1}{(n+2)}\left(\left(1-\dfrac{2\nu}{b_n}\right)f^{\prime}(\nu)+\nu\left(1-\dfrac{\nu}{2b_n}\right) f^{\prime\prime}(\nu)\right)\right]}{(\nu-z)^{p+1}}d\nu\bigg\}\bigg\Vert_r\\
&\geq& \frac{b_n+1}{(n+2)}\bigg\{\bigg\Vert \bigg(\frac{e_1 \left(1-\dfrac{e_1}{2b_n}\right) f^{\prime \prime}+ \left(1-\dfrac{2e_1}{b_n}\right)f^{\prime}}{2}\bigg)^{(p)}\bigg\Vert_r-\dfrac{b_n+1}{(n+2)}\bigg\| \frac{p!}{2\pi}\\&&\times\frac{(n+2)^2}{(b_n+1)^2}\int_{\Gamma}\frac{\bigg(\mathcal{F}_{n}(f;\nu)-f(\nu)-\dfrac{b_n+1}{(n+2)}\left(e_1 \left(1-\dfrac{e_1}{2b_n}\right) f^{\prime \prime}+\left(1-\dfrac{2e_1}{b_n}\right)f^{\prime}\right)\bigg)}{(\nu-z)^{p+1}}d\nu\bigg\|_r \bigg\}.
\end{eqnarray*}
Now, applying Theorem \ref{t2}
\begin{eqnarray*}
\bigg\|\frac{p!}{2\pi}\frac{(n+2)^2}{(b_n+1)^2}\int_{\Gamma}\frac{\bigg(\mathcal{F}_{n}(f;\nu)-f(\nu)-\dfrac{b_n+1}{(n+2)}\left(e_1 \left(1-\dfrac{e_1}{2b_n}\right) f^{\prime \prime}+\left(1-\dfrac{2e_1}{b_n}\right)f^{\prime}\right)\bigg)}{(\nu-z)^{p+1}}d\nu\bigg\|_r\bigg\}
\end{eqnarray*}
\begin{eqnarray*}
&\leq& \frac{p!}{2\pi } \frac{2\pi r_1 \,(n+2)^2}{(b_n+1)^2\,(r_1-r)^{p+1}}\bigg\| \mathcal{F}_{n}(f; \cdot)-f-\frac{b_n+1}{(n+2)}\left(\left(1-\dfrac{2e_1}{b_n}\right)f^{\prime}+e_1 \left(1-\dfrac{e_1}{2b_n}\right)f^{\prime\prime}\right)\bigg\|_{r_1}\\
&\leq&\frac{p! \,\,r_1 \,\,L_{r_1,A}(f)}{(r_1-r)^{p+1}},
\end{eqnarray*}
but by hypothesis on $f,$ we have
$\bigg\|\left(\left(1-\dfrac{2e_1}{b_n}\right)f^{\prime}+ e_1 \left(1-\dfrac{e_1}{2b_n}\right)f^{\prime\prime}\right)^{(p)}\bigg\|_r > 0.$
Indeed if we suppose the contrary that
\begin{eqnarray*}
\bigg\|\left(\left(1-\dfrac{2e_1}{b_n}\right)f^{\prime}+ e_1 \left(1-\dfrac{e_1}{2b_n}\right)f^{\prime\prime}\right)^{(p)}\bigg\|_r =0,
\end{eqnarray*}
then
\begin{eqnarray*}
\left(1-\dfrac{2z}{b_n}\right)f^{\prime}(z)+z \left(1-\dfrac{z}{2b_n}\right)f^{\prime\prime}(z)= Q_{p-1}(z),
\end{eqnarray*}
where $Q_{p-1}(z)$ is a polynomial of degree $\leq p-1,$ thus $f$ satisfies the differential equation
\begin{eqnarray*}
\left(1-\dfrac{2z}{b_n}\right)f^{\prime}(z)+z \left(1-\dfrac{z}{2b_n}\right)f^{\prime\prime}(z)= Q_{p-1}(z), \,\, \forall \,\, |z|\leq r.
\end{eqnarray*}
Now, denoting $f^{\prime}(z)=y(z),$ the above differential equation reduces to
\begin{eqnarray*}
\left(1-\dfrac{2z}{b_n}\right)y(z)+z \left(1-\dfrac{z}{2b_n}\right)y^{\prime}(z)= Q_{p-1}(z), \,\, \forall \,\, |z|\leq r.
\end{eqnarray*}
In what follows, let us define $y(x)=y_1(x)+iy_2(x),$ where $y_1(x)$ and $y_2(x)$ are real functions of the real variable and $i^2=-1$. The functions $y_j(x), j=1,2$ satisfy the differential equations
\begin{eqnarray}\label{m1}
\left(1-\dfrac{2x}{b_n}\right)y_j(x)+x \left(1-\dfrac{x}{2b_n}\right)y_j^{\prime}(x)= Q_{p-1}(x), \,\, \forall \,\, x\in [-1,1],\, \, j=1,2,
\end{eqnarray}
which is a non-homogeneous differential equation. By a similar reasoning as in the proof of Theorem \ref{t3}, the unique solution of homogeneous differential equation corresponding to equation (\ref{m1}) is $y_j(x)=0, \, \forall x\in [-1,1].$ To find the particular solution of non-homogeneous differential equation (\ref{m1}) of the form $y_j(x)=\displaystyle\sum_{k=0}^{p-1} c_k x^k$ with $c_k\in \mathbb{R},$ by simple calculations we easily obtain that
\begin{eqnarray*}
\left(1-\frac{2x}{b_n}\right)\sum_{k=0}^{p-1} c_k x^k+\left(1-\frac{x}{2b_n}\right)\sum_{k=1}^{p-1} k c_k x^k = Q_{p-1}(x)=\sum_{k=0}^{p-1} d_k x^k\,\, \mbox{(say)},
\end{eqnarray*}
which implies that
$$c_0=d_0,\, (k+1)c_k-\left(\dfrac{k-1}{2b_n}+\dfrac{2}{b_n}\right)c_{k-1}=d_k, \,\, k\in \{1,2,\dotsb p-1\}.$$
Thus, $c_k$'s can be uniquely determined.
Hence, it follows that $y_1(x)$ and $y_2(x)$ are polynomials of degree $\leq p-1$ in $x.$ Now, because $y(z)$ is the analytic continuation of $y(x),$ from the identity theorem on analytic functions, it follows that $y(z)$ is a polynomial of degree $\leq \,p-1$ in $z,$ a contradiction to the hypothesis.\\ So,
$$\left\| \left(e_1 \left(1-\dfrac{e_1}{2b_n}\right) f^{\prime\prime}+ \left(1-\dfrac{2e_1}{b_n}\right) f^{\prime}\right)^{(p)}\right\|_r>0.$$
In continuation, reasoning exactly as in the proof of Theorem \ref{t3} and using Theorem \ref{t4}, we get the desired conclusion.
\end{proof}

{\bf Acknowledgements}
The first author is thankful to the "Council of Scientific and Industrial Research" (Grant code: 09/143(0836)/2013-EMR-1) India for financial support to carry out the above research work.

\end{document}